\documentclass[a4paper,11pt]{amsart}
\usepackage{amssymb}
\usepackage{mathrsfs}
\usepackage{amsthm}
\usepackage{amsmath}
\usepackage{latexsym}
\usepackage{graphicx}
\usepackage{numprint}
\usepackage{booktabs}
\usepackage[english]{babel}
\usepackage{fullpage}

\newtheorem{thm}{Theorem}[section]

\newtheorem{lemma}[thm]{Lemma}

\newtheorem{proposition}[thm]{Proposition}

\theoremstyle{definition}
\newtheorem{remark}[thm]{Remark}
\newtheorem{definition}[thm]{Definition}

\numberwithin{equation}{section}

\newcommand{\bQ}{\overline{\mathbb{Q}}}

\newcommand{\bZ}{\overline{\mathbb{Z}}}

\newcommand{\bF}{\overline{\mathbb{F}}}

\newcommand{\C}{\mathbb{C}}
\newcommand{\R}{\mathbb{R}}
\newcommand{\Q}{\mathbb{Q}}

\newcommand{\Z}{\mathbb{Z}}

\newcommand{\F}{\mathbb{F}}
\newcommand{\GSp}{{\rm GSp}}

\newcommand{\lra}{\longrightarrow}

\newcommand{\sgn}{{\rm sgn}}
\newcommand{\A}{\mathbb{A}}

\renewcommand{\O}{\mathcal{O}}

\newcommand{\br}{\overline{\rho}}

\newcommand{\diag}{{\rm diag}}
\newcommand{\disc}{{\rm disc}}
\newcommand{\ds}{\displaystyle}

\newcommand{\uk}{\underline{k}}

\newcommand{\gl}{{\rm GL}}
\newcommand{\GL}{{\rm GL}}

\newcommand{\bs}{\backslash}

\newcommand{\ot}{\overline{\tau}}

\newcommand{\Hom}{{\rm Hom}}

\begin{document}

\title[Automorphy of certain mod 2 Galois representations] 
{Automorphy of mod 2 Galois representations associated to certain genus 2 curves over totally real fields} 
\author{Alexandru Ghitza and Takuya Yamauchi}

\keywords{genus 2 curves, mod 2 Galois representations, automorphy}
\thanks{The second author is partially supported
by JSPS KAKENHI Grant Number (B) No.19H01778.
The authors thank the anonymous referee for their insightful comments.}
\subjclass[2010]{11F, 11F33, 11F80}

\address{Alexandru Ghitza \\
  School of Mathematics and Statistics\\
  University of Melbourne, Parkville, VIC 3010, AUSTRALIA}
\email{aghitza@alum.mit.edu}
\urladdr{https://aghitza.org}

\address{Takuya Yamauchi \\ 
Mathematical Inst. Tohoku Univ.\\
 6-3,Aoba, Aramaki, Aoba-Ku, Sendai 980-8578, JAPAN}
\email{takuya.yamauchi.c3@tohoku.ac.jp}
\urladdr{https://sites.google.com/site/takuyayamauchigsp4}

\maketitle

\begin{abstract}
Given a genus two hyperelliptic curve $C$ over a totally real field $F$,
we show that the mod $2$ Galois representation $\br_{C,2}\colon{\rm Gal}(\overline{F}/F)\lra {\rm GSp}_4(\F_2)$
attached to $C$  
is residually automorphic when the image of $\br_{C,2}$ is isomorphic to $S_5$ and it is also a transitive subgroup under a 
fixed isomorphism ${\rm GSp}_4(\F_2)\simeq S_6$. More precisely, there exists a
Hilbert--Siegel Hecke eigen cusp form $h$ on ${\rm GSp}_4(\A_F)$ of parallel weight two
whose mod $2$ Galois representation $\br_{h,2}$ is isomorphic to $\br_{C,2}$.
\end{abstract}

\bigskip
\section{Introduction}\label{intro}
Let $K$ be a number field in an algebraic closure $\bQ$ of $\Q$ and $p$ be a prime number.  
In proving automorphy of a given geometric $p$-adic Galois representation
\begin{equation*}
  \rho\colon G_K:={\rm Gal}(\bQ/K)\lra {\rm GL}_n(\bQ_p),
\end{equation*}
a first step would be, typically, to observe residual automorphy of its mod $p$ reduction
\begin{equation*}
\br\colon G_K\lra {\rm GL}_n(\bF_p)
\end{equation*}
obtained after choosing a suitable integral lattice.
However, proving residual automorphy of $\br$ is hard unless
the image of $\br$ is reasonable so that one can apply the current results in 
the theory of automorphic representations. For instance, if the image is solvable, one can use the Langlands base change
argument to find an automorphic cuspidal representation as is done in many known cases 
(cf.~\cite{Serre},~\cite{Tu},~\cite{Ca},~\cite{Ra},~\cite{Wong}).
Another natural problem is to find geometric objects whose residual Galois representations have suitably small images in question so that one can apply the 
various known results for automorphy of Galois representations.

In this paper, we study residual automorphy of mod $2$ Galois representations associated to 
certain hyperelliptic curves of genus $2$.
Let us fix some notation to explain our results. 
Let $F$ be a totally real field and fix an embedding $F\subset \bQ$.
Let us consider the hyperelliptic curve over $F$ of genus $2$ defined by
\begin{equation}\label{C}
C\colon\quad y^2=f(x):=x^6+a_1x^5+\dots+a_5x+a_6,\quad a_1,\ldots,a_6\in F.
\end{equation}
Let $D_f$ be the discriminant of $f$. If we write $f(x)=\ds\prod_{i=1}^6(x-\alpha_i)$ 
over $\bQ$, then  $D_f=\ds\prod_{1\le i<j\le 6}(\alpha_i-\alpha_j)^2$. 
By abusing the notation, we also denote by $C$ a unique smooth completion of the 
above hyperelliptic curve. 
Let $J={\rm Jac}(C)$ be the Jacobian variety of $C$ and for each positive integer $n$, let $J[n]$ be the group scheme of
$n$-torsion points. Let $T_{J,2}$ 
be the $2$-adic Tate module over $\Z_2$ associated to $J$.
Let $\langle \ast,\ast \rangle\colon T_{J,2}\times T_{J,2}\lra \Z_2(1)$ be the Weil pairing, which is $G_F$-equivariant, perfect, and alternating.
It yields an integral $2$-adic Galois representation
\[
  \rho_{C,2}\colon G_F\lra {\rm GSp}(T_{J,2},\langle \ast,\ast \rangle)\simeq {\rm GSp}_4(\Z_2)
\]
where the algebraic group $\mathrm{GSp}_4= \mathrm{GSp}_J$ is the symplectic similitude group in $\mathrm{GL}_4$ 
associated to $\begin{pmatrix} 0_2& s\\-s &0_2\end{pmatrix},\ 
s=\begin{pmatrix} 0& 1\\1 &0\end{pmatrix}$. 
Put
\begin{equation*}
\overline{T}_{J,2}=T_{J,2}\otimes_{\Z_2}\F_2=J[2](\bQ).
\end{equation*}
This yields a mod 2 Galois representation 
\begin{equation}\label{eq:mod2}
  \br_{C,2}\colon G_F\lra 
{\rm GSp}(\overline{T}_{J,2},\langle \ast,\ast \rangle_{\F_2})\simeq {\rm GSp}_4(\F_2).
\end{equation} 
On the other hand, we can embed the Galois group ${\rm Gal}(F_{f}/F)$ of the splitting field of $f$ over $F$
into $S_6$ by permutation of the roots of $f(x)$; we denote this embedding by $\iota_f\colon {\rm Gal}(F_{f}/F)\hookrightarrow S_6$.
Here $S_n$ stands for the $n$-th symmetric group.
Since $\overline{T}_{J,2}$ is generated (cf.~\cite{Sa}) by divisors associated
to the points $(x,0)$ in $C(\bQ)$,  $\br_{C,2}$ factors through ${\rm Gal}(F_{f}/F)$; we denote it by
$\br_{C,2}\colon {\rm Gal}(F_{f}/F)\hookrightarrow {\rm GSp}_4(\F_2)$ again. A subgroup $H$ of $S_6$ is transitive if 
it acts on the set $\{1,2,3,4,5,6\}$ transitively. 
Let us fix an isomorphism ${\rm GSp}_4(\F_2)\simeq S_6$ commuting with $\br_{C,2}$ and $\iota_f$. 

In this situation we will prove the following:
\begin{thm}\label{main1}
Keep the notation as above. Suppose that ${\rm Im }(\br_{C,2})\simeq S_5$ and 
it is a transitive subgroup of $S_6$ via
the fixed isomorphism between ${\rm GSp}_4(\F_2)$ and $S_6$. Further assume that for any complex conjugation $c$ in $G_F$, 
$\br_{C,2}(c)$ is of type $(2,2)$ as an element of $S_5$.  
Then 
there exists a Hilbert--Siegel Hecke eigen cusp form $h$ on ${\rm GSp}_4(\A_F)$ of
parallel weight $2$ such that $\br_{h,2}\simeq \br_{C,2}$
as a representation to ${\rm GL}_4(\bF_2)$,
where $\br_{h,2}$ is the reduction modulo $2$ of the $2$-adic representation $\rho_{h,2}$ associated to $h$ $($see~Section \ref{agr} for $\rho_{h,2})$.  
\end{thm}
We give some remarks on the above theorem.
The transitivity of the image implies $f(x)$ is irreducible over $F$.  
Put $L=F(\sqrt{D_f})$. It is corresponding to the kernel of 
$\sgn\circ\iota_f\circ\br_{C,2}\colon G_F\lra \{\pm1\}$ where $\sgn\colon S_6\lra\{\pm1\}$ is the usual 
sign character. 
The assumption on $\br_{C,2}(c)$ shows 
$f(x)$ has only two real roots and the other four complex roots are permuted by an element of 
type $(2,2)$ in $S_6$. Therefore, $D_f\in F$ is totally real and it is not a square element 
since ${\rm Im }(\br_{C,2})\simeq S_5$. 
In conclusion, $L/F$ is a totally real quadratic extension. 
Then it will turn out that $\br_{C,2}$ is an induced representation to $G_F$ of 
a certain totally odd $2$-dimensional mod $2$ Galois representation $\ot\colon G_L\lra {\rm GL}_2(\bF_2)$ whose image is isomorphic to $A_5$.  Applying Sasaki's result~\cite{Sasaki} 
(or Pilloni--Stroh's result \cite{PS}),
the Jacquet--Langlands correspondence, and a suitable congruence method, one can find a
Hilbert cusp form $g$ of parallel weight 2 on 
${\rm GL}_2(\A_L)$ whose corresponding mod 2 Galois representation $\br_{g,2}$ is isomorphic to $\ot$. 
Then the theta lift of the corresponding automorphic representation $\pi_g$ to ${\rm GSp}_4(\A_F)$ 
yields a cuspidal automorphic representation $\Pi$ of ${\rm GSp}_4(\A_F)$ which is generated by a Hilbert--Siegel Hecke eigen cusp form $h$ of parallel weight $2$ on ${\rm GSp}_4(\A_F)$. By Theorem 1.1 of \cite{Weiss},  for each 
$p$ in a set of primes of Dirichlet density one, the Hodge--Tate weights at $p$ of the 
$p$-adic Galois representation attached to $\Pi$ are $\{0,0,1,1\}$. It is also true for $p=2$ when $\Pi_2$ is unramified and the roots of the $2$-nd Hecke polynomial of $\Pi$ 
are pairwise distinct (see Theorem 3.3 of \cite{Weiss}). 
On the other hand, the Hodge--Tate weights of $\rho_{C,2}$ at the place dividing $2$  are always $\{0,0,1,1\}$ and this is why we seek a form $h$ of such weight. 
 
We note that $J$ is potentially automorphic by Theorem 1.1.3 of~\cite{BCGP} and the potential version of Theorem~\ref{main1} is already known in a more general setting.
\begin{remark}\label{level-GT}
\begin{enumerate}
\item The construction of Hilbert-Siegel modular forms in Theorem \ref{main1} is 
similar to the one in \cite{TY}. However, to apply theta lifting to $\GSp_4$, we need to 
carefully look at the central characters after a congruence method 
$($see Proposition \ref{cong}$)$.
\item In the main theorem, we do not specify the levels of the Hilbert--Siegel forms  because of the lack of level-lowering results $($see Remark~\ref{level}$)$. 
However, $h$ could be a paramodular form by using the method of \cite{JR} 
at least when $F=\Q$ so that the situation is compatible with the paramodular 
conjecture \cite[Conjecture 1.4]{BK}. 
\item 
In the course of the proof for the main theorem, we do not use the results in~\cite{GeeT} to associate a form with a Galois representation though it is necessary to state Theorem~\ref{gal}.  Instead, we apply the unconditional result in~\cite{R} to
construct a corresponding automorphic cuspidal representation. Hence Theorem~\ref{main1} is true unconditionally. 
\item The central character of $h$ can be read off from 
Theorem \ref{gal-theta} and Proposition \ref{cong}. 
\item 
We can weaken the condition on the image of $\br_{C,2}$ by requiring that
the image be isomorphic to 
$F_{20}:=C_4\ltimes C_5$, or isomorphic to $A_5$ with $F/\Q$ of even degree. 
In the $F_{20}$ case, we have a similar cuspidal representation.
In the $A_5$ case, $\br_{C,2}$ is reducible so we can naively attach a non-cuspidal automorphic representation.
By using a congruence method, we may hope to obtain
a cuspidal representation in either case but we do not pursue it in this article. 
\end{enumerate}
\end{remark}

This paper is organized as follows.
In Section~\ref{sect:mod2}, we consider some basic facts about mod 2 Galois representations to ${\rm GSp}_4(\F_2)$.
We devote Section~\ref{sect:auto} to the study of the automorphic forms in question and to various congruences between
several types automorphic forms, and apply this to the proof of automorphy of our representations in~\ref{auto}.
Finally, in Section~\ref{sect:examples} we describe how to obtain explicit families of examples using two constructions: the first coming from a classical result of Hermite, and the second from $5$-division points on elliptic curves following Goins~\cite{G-thesis}.

\section{Certain mod 2 Galois representations to $\mathrm{GSp}_4(\F_2)$}\label{sect:mod2}
In this section, we study some properties of mod $2$ Galois representations to
${\rm GSp}_4(\F_2)$. 
We refer to~\cite[Section 5]{Bet} and \cite[Section 3]{TY}.
We denote by $S_n$ the $n$-th symmetric group, by $A_n$ the $n$-th alternating group, and by $C_n$ the cyclic group of order $n$.
\subsection{${\rm GSp}_4$}
Let us introduce the smooth group scheme 
$\mathrm{GSp}_4=\mathrm{GSp}_J$ over $\Z$ which is defined 
as the symplectic similitude group in $\mathrm{GL}_4$ associated to $J:=\begin{pmatrix} 0_2& s\\-s &0_2\end{pmatrix},\
s=\begin{pmatrix} 0& 1\\1 &0\end{pmatrix}$. Explicitly,
\[
  \mathrm{GSp}_4=\left\{X\in \mathrm{GL}_4\ |\ {}^{t}XJX=\nu(X)J,\ \exists \nu(X)\in \mathrm{GL}_1\right\}.
\]
Put $\mathrm{Sp}_4={\rm Ker}(\nu\colon\mathrm{GSp}_4\lra \mathrm{GL}_1,X\mapsto \nu(X))$. 

\subsection{An identification between ${\rm GSp}_4(\F_2)$ and $S_6$}\label{GSp4vsS6}
Let $s\colon\F^6_2\lra \F_2$ be the linear functional defined by $s(x_1,\ldots,x_6)=x_1+\cdots+x_6$. Put $V=\{x\in \F^6_2\ |\ s(x)=0\}$ and $W=V/U$ where $U=\langle (1,1,1,1,1,1) \rangle$. 
Let us consider the bilinear form on $\F^6_2$ given by the formula
\[
  \langle x,y \rangle=x_1y_1+\cdots +x_6y_6,\quad x,y\in \F^6_2.
\]
It induces a non-degenerate, alternating pairing $\langle \ast,\ast \rangle_W$ on $W$, where
being alternating means that $\langle x,x \rangle_W=0$ for each $x\in W$.
The symmetric group $S_6$ acts naturally on $\F^6_2$ and it yields a group homomorphism
\[
  \varphi\colon S_6\lra {\rm GSp}_{\F_2}(W,\langle \ast,\ast \rangle_W)\simeq  {\rm GSp}_4(\F_2).
\]
The action of $S_6$ on $W$ is faithful. In fact, we can check it only for 
$(12)$ and $(123456)$ generating $S_6$. By direct computation, 
\begin{equation*}
|{\rm GSp}_4(\F_2)|={\rm Sp}_4(\F_2)=720=|S_6|
\end{equation*}
and therefore, $\varphi$ is an isomorphism. 

It is easy to find a basis of $W$. 
For instance, we see that 
\begin{equation}\label{basis}
\begin{aligned}
  e_1=(1,1,0,0,0,0)&,\ e_2=(0,0,0,1,1,0), \\ 
  e_3=(0,0,0,1,0,1)&,\ e_4=(1,0,1,0,0,0)
\end{aligned}
\end{equation}
form a basis of $W$. The representation matrices with respect to the above basis 
for the generators $(12)$ and $(123456)$ of $S_6$ are given respectively as follows:
\begin{equation}\label{mrep}
\left(
\begin{array}{cccc}
 1 & 0 & 0 & 1 \\
 0 & 1 & 0 & 0 \\
 0 & 0 & 1 & 0 \\
 0 & 0 & 0 & 1 \\
\end{array}
\right),\ 
\left(
\begin{array}{cccc}
 1 & 1 & 1 & 0 \\
 1 & 1 & 0 & 1 \\
 0 & 1 & 1 & 0 \\
 1 & 0 & 0 & 1 \\
\end{array}
\right).
\end{equation}
\subsection{$S_5$-representations to ${\rm GSp}_4(\F_2)$} 
It follows from ~\cite[Section 5]{Bet} (in particular, the table in Lemma 5.1.7 in loc.cit.) 
that, up to conjugacy, there are two embeddings from $S_5$ to $\GSp_4(\F_2)\simeq S_6$ which are absolutely 
irreducible as representations to $\gl_4(\bF_2)$. 
We denote the images by $S_5(a), S_5(b)$ respectively; they are characterized in terms of the trace as follows:
\begin{enumerate}
\item $($type $S_5(a))$ the elements of order $3$ have trace $0$ and are of type $(3,3)$,
\item $($type $S_5(b))$ the elements of order $3$ have trace $1$ and are of type $(3)$.
\end{enumerate}
Under the identification $\GSp_4(\F_2)\simeq S_6$, these subgroups are given explicitly as 
\begin{equation}\label{exp-form}
  \begin{aligned}
    S_5(a)&=\big\langle \sigma_6:=(12346),\ \sigma_{2^3}:=(12)(34)(56) \big\rangle,\\
    S_5(b)&=\{\sigma\in S_6\ |\ \sigma(6)=6\}.
  \end{aligned}
\end{equation}
Notice that $\sigma_{2^3}(\sigma_6\sigma_{2^3}\sigma^{-1}_6)$ is of type $(3,3)$ and 
$(\sigma_6\sigma_{2^3})^2=(16)(35)$ is obviously of type $(2,2)$ which is explicitly given 
by $\left(
\begin{array}{cccc}
 0 & 1 & 1 & 0 \\
 1 & 1 & 1 & 1 \\
 0 & 1 & 1 & 1 \\
 1 & 0 & 1 & 0 \\
\end{array}
\right)$ with respect to the basis (\ref{basis}). By direct computation, one can check 
it is conjugate to $J$ by $\left(
\begin{array}{cccc}
 1 & 0 & 1 & 0 \\
 0 & 1 & 1 & 1 \\
 0 & 0 & 1 & 0 \\
 0 & 0 & 0 & 1 \\
\end{array}
\right)\in \GSp_4(\F_2)$. 

Let $G$ be a group and $\br\colon G\lra \GSp_4(\F_2)$ a representation of $G$. 
Since any element of $\GSp_4(\F_2)={\rm Sp}_4(\F_2)$ has determinant one, 
$\br$ is automatically self-dual. 
\begin{lemma}\label{rep-S5}Assume that ${\rm Im}(\br)\simeq S_5$. Then $\br$ is absolutely irreducible. 
Furthermore, there exist an index $2$ subgroup $H$ of $G$, a generator $\iota$ of $G/H$, and an absolutely irreducible representation
$\ot\colon H\lra {\rm GL}_2(\F_4)$ which is not equivalent to ${}^\iota \ot$ 
defined by ${}^\iota \ot(h)=\ot(\iota h\iota^{-1})$ for each $h\in H$  such that 
\begin{enumerate}
\item ${\rm Im}(\ot)={\rm SL}_2(\F_4)\simeq A_5${\rm ;}
\item if ${\rm Im}(\br)$ is of type $S_5(a)$, then $\br\simeq {\rm Ind}^G_H\ot${\rm ;}
\item if ${\rm Im}(\br)$ is of type $S_5(b)$, then $\br$ is isomorphic to  
a twisted tensor product with respect to $G/H$ $($see~\cite[Section 3]{TY}$)$.
\end{enumerate}
\end{lemma}
\begin{proof}
Put $H:=\br^{-1}(A_5)$. Clearly it is a subgroup of $G$ of index $2$.
Since ${\rm Im}(\br)\simeq S_5$, by Dickson's result \cite[p.128,~(2.1)]{Wagner}, 
$\br$ is absolutely irreducible. 
Then by Clifford's theorem, it follows that $\br|_H$ is either 
absolutely irreducible (by Dickson's result again) or 
a sum of two 2-dimensional irreducible representations $\ot,\ot'\colon G\lra \gl_2(\bF_2)$ 
after base extension to $\bF_2$. 
In the latter case, by
Schur's Lemma and Frobenius reciprocity,
$$1\le\dim\Hom_H(\br|_H,\ot)=\dim\Hom_G(\br,{\rm Ind}^G_H\ot)\le 1,$$
therefore both have the same dimension and $\br$ is irreducible.
Further, irreducibility implies that $\tau$ cannot be isomorphic to ${}^\iota \ot$.  
Since $A_5\simeq {\rm SL}_2(\F_4)$ (cf. \cite[Section 3.3]{TY}), 
we may assume $\ot$ takes values in $\F_4$ as desired.  
Further, by using (\ref{mrep}) and (\ref{exp-form}), 
one can distinguish the types in terms of the trace. 

The remaining (absolutely irreducible) case follows from~\cite[Proposition 3.5-(1)]{TY}.
\end{proof}
Fix an isomorphism ${\rm GSp}_4(\F_2)\simeq S_6$ as in Section \ref{GSp4vsS6}. 
\begin{remark}\label{types}Keep the notation as above $($recall that ${\rm Im}(\br)\simeq S_5)$.
The type of the image ${\rm Im}(\br)$ can also be characterized in the following way:
\begin{enumerate}
\item $($type $S_5(a))$ it is a transitive subgroup of $S_6$;
\item $($type $S_5(b))$ it is not transitive, hence up to conjugacy, it fixes $6$ as a subgroup of $S_6$.
\end{enumerate}
\end{remark}
In this article, we focus on the case of type $S_5(a)$; the other case is 
studied in~\cite{TY}.
\begin{proposition}\label{s5Galois}Let $F$ be a totally real field. Let $\br\colon G_F\lra {\rm GSp}_4(\F_2)$ be an irreducible  
mod 2 Galois representation. Suppose that ${\rm Im}(\br)$ is of type $S_5(a)$.  
Assume further that for each complex conjugation $c$ of $G_F$, $\br(c)$ is of type $(2,2)$. 
Then there exist a totally real quadratic extension $L/F$ with 
${\rm Gal}(L/F)=\langle \iota \rangle$  and an irreducible totally odd Galois 
representation $\ot\colon G_L\lra {\rm GL}_2(\F_4)$ which is not equivalent to ${}^\iota\ot$ satisfying 
\begin{enumerate}
\item ${\rm Im}(\ot)\simeq {\rm SL}_2(\F_4)\simeq A_5$;
\item $\br \simeq {\rm Ind}^{G_F}_{G_L}\ot$ 
as a representation to ${\rm GL}_4(\F_4)$;
\item  for each complex conjugation $c$ of $G_F$, $\ot(c)$ is conjugate to 
$s=\begin{pmatrix}
0 & 1  \\
1 & 0 
\end{pmatrix}$.
\end{enumerate}
Here the totally odd-ness for $\ot$ exactly means the third condition above.
\end{proposition}
\begin{proof}
Let $L/F$ be the quadratic extension corresponding to 
the sign character $\sgn\colon {\rm Im}(\br)\simeq S_5\lra \{\pm 1\}$.  
By assumption $L/F$ is a totally real quadratic extension and we have  
$\br(G_L)\simeq A_5$ (see the explanation right after Theorem \ref{main1}). 
The first two claims follow from Lemma~\ref{rep-S5}.
Further, by assumption, for each complex conjugation $c$ of $G_F$, 
$\br(c)$ is conjugate to $J$ (see the paragraph right after (\ref{exp-form})).
The third claim follows from this. 
\end{proof}

\section{Automorphy}\label{sect:auto}
Fix an embedding $\bQ\hookrightarrow \C$ and an isomorphism $\bQ_p\simeq \C$ 
for each prime $p$. 
Let $\iota=\iota_p\colon \bQ\lra\bQ_p$ be an embedding that is compatible with the maps  $\bQ\hookrightarrow \C$ and
$\bQ_p\simeq \C$ fixed above.
\subsection{Automorphic Galois representations}\label{agr}
Let $F$ be a totally real field. For each place $v$ of $F$, let $F_v$ be the completion of $F$ at $v$.
In this section we recall basic properties of cuspidal automorphic 
representations of ${\rm GSp}_4(\A_F)$ such that each infinite component is a 
(limit of) discrete series representation. We basically follow the notation of Mok's article~\cite{Mok} and add ingredients necessary for our purpose.

For any place $v$ of $F$, we denote by $W_{F_v}$ the Weil group of $F_v$. 
Let $m_1,m_2,w$ be integers such that
\begin{equation*}
  m_1>m_2\ge 0 \quad\text{and}\quad m_1+m_2\equiv w+1\pmod{2}.
\end{equation*}
Let $a=(m_1+m_2)/2$, $b=(m_1-m_2)/2$, and consider the $L$-parameter $\phi_{(w;m_1,m_2)}\colon W_\R\lra {\rm GSp}_4(\C)$ defined by
\begin{equation}\label{type}
\phi_{(w;m_1,m_2)}(z)=|z|^{-w}\diag\left(
\Big(\frac{z}{\overline{z}}\Big)^{a},
\Big(\frac{z}{\overline{z}}\Big)^{b},
\Big(\frac{z}{\overline{z}}\Big)^{-b},
\Big(\frac{z}{\overline{z}}\Big)^{-a}\right)
\end{equation}
and 
\[
\phi_{(w;m_1,m_2)}(j)=
\left(\begin{array}{cc}
0_2 & s \\
(-1)^{w+1} s & 0_2
\end{array}
\right)
,\quad s=\left(\begin{array}{cc}
0 & 1 \\
1 & 0
\end{array}
\right).
\]
The archimedean $L$-packet $\Pi(\phi_{(w;m_1,m_2)})$ corresponding to $\phi_{(w;m_1,m_2)}$ under the Local Langlands Correspondence
consists of two elements 
\begin{equation*}
\left\{\pi^H_{(w;m_1,m_2)},\ \pi^W_{(w;m_1,m_2)}\right\} 
\end{equation*}
whose central characters both satisfy $z\mapsto z^{-w}$ for $z\in \R^\times_{>0}$.  
These are essentially tempered unitary representations of ${\rm GSp}_4(\R)$ and tempered exactly when $w=0$. Let $K$ be a maximal compact subgroup of ${\rm Sp}_4(\R)$. 
When $m_2\ge 1$ (resp. $m_2=0$) the representation $\pi^H_{(w;m_1,m_2)}$ is called a (resp.\ limit of) holomorphic discrete series representation of
minimal $K$-type $\uk=(k_1,k_2):=(m_1+1,m_2+2)$ which corresponds to an algebraic representation 
$V_{\uk}:={\rm Sym}^{k_1-k_2}{\rm St}_2\otimes \det^{k_2}{\rm St}_2$ of $K_\C={\rm GL}_2(\C)$ where ${\rm St}_2$ stands for the 2-dimensional standard algebraic representation 
of $K_\C$. 
If $m_2\ge 1$ (resp. $m_2=0$), the representation $\pi^W_{(w;m_1,m_2)}$ is  called a (resp.\ limit of) large (or generic) discrete series representation of minimal $K$-type $V_{(m_1+1,-m_2)}$. 
We say that a cuspidal automorphic representation $\pi$ of ${\rm GSp}_4(\A_F)$ is 
regular if for each infinite place $v$, $\pi_v$ is a discrete series representation. 

Fix an integer $w$.  
Let $\pi=\otimes'_{v}\pi_v$ be an automorphic cuspidal representation of ${\rm GSp}_4(\A_F)$ such that 
for each infinite place $v$, $\pi_v$ has $L$-parameter $\phi_{(w;m_{1,v},m_{2,v})}$ with
the parity condition $m_{1,v}+m_{2,v}\equiv w+1\pmod{2}$. Let ${\rm Ram}(\pi)$ be the set of all
finite places at which $\pi_v$ is ramified. 

For the existence of Galois representations attached to 
regular algebraic self-dual cuspidal representations for general linear groups, 
see Theorem 1.1 of \cite{BGHT}, which is the main reference for knowing the main contributors
(Harris--Taylor, Shin, and others) as well as the latest results up to the time when loc.\ cit.\ appeared.
In order to attach Galois representations for $\rm GSp_4$, both Sorensen~\cite{So} and Mok~\cite{Mok}
use Langlands functoriality from $\rm GSp_4$ to $\rm GL_4$ and reduce the problem 
to the known cases for general linear groups (specifically $\rm GL_4$). 
Sorensen focused on the case when the representation at any infinite place is 
generic (large discrete series) while Mok treated the limit of discrete series via a congruence method.
Mok relied on results of Arthur that were at the time expected but not yet proved, and which are now 
guaranteed by~\cite{GeeT} though remaining conditional on Arthur's classification. 
We summarize here the known results taken from Theorem 3.1 
(for the regular case) and 
Theorem 4.14 of~\cite{Mok} 
(see also Theorem 3.1 and Theorem 3.3 of \cite{Weiss} as a good summary): 
\begin{thm}\label{gal}
Assume that $\pi$ is neither CAP nor endoscopic. 
For each prime $p$ and $\iota_p\colon \bQ_p\stackrel{\sim}{\lra} \C$ 
there exists a continuous, semisimple Galois representation 
$\rho_{\pi,\iota_p}\colon G_F\lra {\rm GSp_4}(\bQ_p)$ such that 
\begin{enumerate}
\item $\nu\circ \rho_{\pi,\iota_p}(c_\infty)=-1$ for any complex conjugation $c_\infty$ in $G_F$. 

\item $\rho_{\pi,\iota_p}$ is unramified at all finite places that do not belong to the set ${\rm Ram}(\pi)\cup\{v|p\}$;

\item for each finite place $v$ not lying over $p$, the local-global compatibility holds:
  \[
  {\rm rec}^{{\rm GT}}_v\left(\pi_v\otimes |\nu|^{-\frac{3}{2}}\right)\simeq 
    \begin{cases}
  {\rm WD}\left(\rho_{\pi,\iota_p}|_{G_{F_v}}\right)^{F-{\rm ss}} & \text{if $\pi$ is regular,} \\
   {\rm WD}\left(\rho_{\pi,\iota_p}|_{G_{F_v}}\right)^{{\rm ss}} & \text{otherwise,}
    \end{cases}
  \]
with respect to $\iota_p$, where ${\rm rec}^{{\rm GT}}_v$ stands for the local Langlands correspondence
constructed by Gan--Takeda~\cite{GT};

\item for each $v|p$ and each embedding $\sigma\colon F_v\hookrightarrow \bQ_p$, let
$v_\sigma\colon F\hookrightarrow \C$ denote the embedding $v_\sigma:=\iota_p\circ \sigma|_F$. 
Then the representation $\rho_{\pi,\iota_p}|_{G_{F_v}}$ has Hodge--Tate weights
\[
HT_\sigma\left(\rho_{\pi,\iota_p}|_{G_{F_v}}\right)=\left\{\delta_{v_\sigma},\delta_{v_\sigma}+m_{2,v_\sigma},
\delta_{v_\sigma}+m_{1,v_\sigma},\delta_{v_\sigma}+m_{2,v_\sigma}+m_{1,v_\sigma} \right\}
\]
where $\delta_{v_\sigma}=\ds\frac{1}{2}(w+3-m_{1,v_\sigma}-m_{2,v_\sigma})$; 
\item further, 
\begin{enumerate}
\item if $\pi_{v'}$ is discrete series for all infinite places $v'$, then 
for each finite place $v$ of $F$, 
$\rho_{\pi,\iota_p}|_{G_{F_v}}$ is of de Rham and the local-global compatibility also holds 
up to Frobenius semi-simplification; 
\item for each finite place $v$ such that $\pi_v$ is unramified and the Satake parameters 
of $\pi$ are distinct from each other, $\rho_{\pi,\iota_p}|_{G_{F_v}}$ is crystalline and the local-global compatibility also holds up to semi-simplification.
\end{enumerate}
\end{enumerate}
\end{thm}
\begin{remark}\label{rmk-gal}
\begin{enumerate}
\item Regarding the fourth statement of the above theorem, when $\pi$ is not regular, 
in general, we do not even know if $\rho_{\pi,\iota_p}|_{G_{F_v}}$ 
is Hodge-Tate for each $v|p$ such that $\pi_{v_\sigma}$ is not discrete series. 
We remark that the case $($b$)$ of the fifth claim follows from \cite[Theorem 3.3-(vi)]{Weiss}.  
{\rm (}see Remark 3.4 of \cite{Weiss}{\rm )}. 
\item A priori, by construction, $\rho_{\pi,\iota_p}$ takes values in $\GL_4(\bQ_p)$. 
However, by \cite[Corollary 1.3]{BC}, it factors through $\GSp_4(\bQ_p)$ as stated.  
\item Notice $3-k_1-k_2=-(m_1+m_2)$. Regarding Theorem \ref{gal}-(3),  
if we put $w=0$  and consider the twist 
$\pi_v\otimes |\nu|^{-\frac{3}{2}}\otimes |\nu|^{\delta_v}=\pi_v\otimes |\nu|^{\frac{3-k_1-k_2}{2}}$ which yields the setting in \cite[Theorem 3.1-(vii)]{Weiss}. 
Accordingly, our Hodge-Tate weights are also shifted by $-\delta_{v_\sigma}$ and they 
give the Hodge-Tate weights in \cite[Theorem 3.1-(v)]{Weiss}.  
\end{enumerate}
\end{remark}

\begin{definition}\label{auto-gal} 
\begin{enumerate}
\item
Let $\rho\colon G_F\lra {\rm GSp}_4(\bQ_p)$ be a $p$-adic Galois representation. 
We say $\rho$ is automorphic if there exists a cuspidal automorphic representation $\pi$ of 
${\rm GSp}_4(\A_F)$ with $\pi_v$ a $($limit of$)$ discrete series representation for any $v|\infty$, and such that
$\rho\simeq\rho_{\pi,\iota_p}$ as a representation to ${\rm GL}_4(\bQ_p)$. 
\item  Let $\br\colon G_F\lra {\rm GSp}_4(\bF_p)$ be an irreducible mod $p$  Galois representation. 
We say $\br$ is automorphic if there exists a cuspidal automorphic representation $\pi$ of 
${\rm GSp}_4(\A_F)$ with $\pi_v$ a $($limit of$)$ discrete series representation for any $v|\infty$, and such that
$\br\simeq \br_{\pi,\iota_p}$ as a representation to ${\rm GL}_4(\bF_p)$.
\end{enumerate}
\end{definition} 
\begin{remark}\label{switch}
Let $h$ be a holomorphic Hilbert--Siegel Hecke eigen cusp form on ${\rm GSp}_4(\A_F)$ of parallel weight $2$.
Let $\pi_h$ be the corresponding cuspidal automorphic representation of ${\rm GSp}_4(\A_F)$. 
Then for each infinite place $v$, the local Langlands parameter at $v$ is given by $\phi_{(w;1,0)}$ and 
$\pi_{h,\infty}$ is a limit of holomorphic discrete series. 
Conversely, if a cuspidal automorphic representation $\pi$ of $\GSp_4(\A_F)$ is neither CAP nor endoscopic and 
its local Langlands parameter at $v|\infty$ is given by $\phi_{(w;1,0)}$ for some integer $w$, 
one can associate such a form $h$ by using~\cite{GeeT}.
Note that the results in~\cite{GeeT} are conditional on the trace formula
$($see the second paragraph on p.472 of~\cite{GeeT}$)$.
However, in the course of the proof of Theorem~\ref{main1}, we apply the unconditional result in~\cite{R} to
construct a corresponding automorphic cuspidal representation. Hence Theorem~\ref{main1} is true unconditionally.
\end{remark}

\subsection{Hilbert modular forms and the Jacquet--Langlands correspondence}\label{PNPJL}
We refer to~\cite[Section 1]{Taylor1} and~\cite[Section 3]{Kisin} for the theory of
($p$-adic) algebraic modular forms corresponding to (paritious) Hilbert modular forms via the Jacquet--Langlands correspondence.

In this section, $p$ is any rational prime but we remind the reader that we will later 
consider $p=2$.
Let $M$ be a totally real field of even degree $m$. 
For each finite place $v$ of $M$, let $M_v$ be 
the completion of $M$ at $v$, $\O_v$ its ring of integers, $\varpi_v$ a uniformizer of $M_v$, and $\F_v$ the
residue field of $M_v$.   
Let $D$ be the quaternion algebra with center $M$ which is
ramified exactly at all the infinite places of $M$ and $\O_D$ be the ring of integral quaternions of $D$.
For each finite place $v$ of $M$, we fix an 
isomorphism
\begin{equation*}
\iota_v\colon D_v:=D\otimes_{M}M_v\simeq {\rm M}_2(M_v).
\end{equation*}
We view $D^\times$ as an algebraic group over $M$ so that for any $M$-algebra $A$, $D^\times(A)$ outputs 
$(D\otimes _{M}A)^\times$ and similarly as an algebraic group scheme over $\O_M$ such that
$D^\times (R)=(\O_D\otimes_{\O_M}R)^\times$ for any $\O_M$-algebra $R$.  

Let $K$ be a finite extension of $\Q_p$ contained in $\bQ_p$ with residue field $k$ and ring of integers $\O_K$, and
assume that $K$ contains the images of all embeddings $M\hookrightarrow \bQ_p$.

For each finite place $v$ of $M$ lying over $p$, let $\tau_v$ be  a smooth representation of
${\rm GL}_2(\mathcal{O}_{v})$ acting on a finite free 
$\O_K$-module $W_{\tau_v}$. We also view it as a representation of $D^\times_v$ via 
$\iota_v$.
Put $\tau:=\ds\otimes_{v|p}\tau_v$ which is a representation of ${\rm GL}_2(\mathcal{O}_{p}):=\prod_{v|p}\GL_2(\O_v)$ 
acting on $W_\tau:=\otimes_{v|p}W_{\tau_v}$. Suppose 
$\psi\colon M^\times\bs (\A^\infty_M)^\times\lra \mathcal{O}^\times_K$ is a continuous character so that 
for each $v|p$, $Z_{D^\times}(\mathcal{O}_{v})\simeq \mathcal{O}^\times_{v}$ acts on 
$W_{\tau_v}$ by $\psi^{-1}|_{\mathcal{O}^\times_{M_v}}$ where $Z_{D^\times}\simeq GL_1$ is the center of $D^\times$
as a group scheme over $\O_M$. Note that we put the discrete topology on 
$\mathcal{O}^\times_K$ and 
therefore, such a character is necessarily of finite order. Let $U=\prod_v U_v$ be a compact open subgroup 
of $D^\times(\A^\infty_{M})\simeq {\rm GL}_2(\A^\infty_{M})$ such that $U_v\subset D^\times({O}_{M_v})$ 
for all finite places $v$ of $M$.
Put $U_p:=\prod_{v|p}U_v$ and $U^{(p)}=\prod_{v\nmid p}U_v$. 
For any local $\mathcal{O}_K$-algebra $A$ put $W_{\tau,A}:=
W_\tau\otimes_{\mathcal{O}_K} A$. Let 
$\Sigma$ be a finite set of finite places of $M$.
For each $v\in \Sigma$, let $\chi_v\colon U_v\lra A^\times$ be a quasi-character. Define 
$\chi_{\Sigma}\colon U\lra A^\times$ whose local component is $\chi_v$ if $v\in \Sigma$, the trivial representation otherwise.

\begin{definition}[$p$-adic algebraic quaternionic forms]\label{dfn-2AMF}
Let $S_{\tau,\psi}(U,A)$ denote the space of functions
\begin{equation*}
f\colon D^\times\bs D^\times(\A^\infty_{M}) \lra W_{\tau,A}
\end{equation*}
such that 
\begin{itemize}
\item $f(gu)=\tau(u_p)^{-1}f(g)$ for $u=(u^{(p)},u_p)\in U=U^{(p)}\times U_p$ and any $g\in D^\times(\A^\infty_{M})$;
\item $f(zg)=\psi(z)f(g)$ for $z\in Z_{D^\times}(\A^\infty_M)$ and $g\in D^\times(\A^\infty_{M})$.
\end{itemize}
Similarly, 
let $S_{\tau,\psi,\chi_{\Sigma}}(U,A)$ denote the space of functions
\begin{equation*}
f\colon D^\times\bs D^\times(\A^\infty_{M})\lra W_{\tau,A}
\end{equation*}
such that 
\begin{itemize}
\item $f(gu)=\chi^{-1}_{\Sigma}(u)\tau(u_p)^{-1}f(g)$ for $u=(u^{(p)},u_p)\in U=U^{(p)}\times U_p$ and 
$g\in D^\times(\A^\infty_{M})$;
\item $f(zg)=\psi(z)f(g)$ for $z\in Z_{D^\times}(\A^\infty_M)$ and $g\in D^\times(\A^\infty_{M})$.
\end{itemize}
We call a function belonging to these spaces a $p$-adic algebraic quaternionic form.
\end{definition} 

Let $S$ be a finite set of finite places of $M$ containing 
all places $v\nmid p$ such that $U_v\neq D^\times(\mathcal{O}_{v})$.  
We define the (formal) Hecke algebra 
\begin{equation}\label{fha1}
\mathbb{T}^S_A:=A[T_{v},S_v]_{v\not\in S\cup\{v|p\}}
\end{equation}
where
\begin{align*}
  T_{v}&=[D^\times(\mathcal{O}_{v})\iota^{-1}_{v}(\diag(\pi_v,1))D^\times(\mathcal{O}_{v})],\\
  S_v&=[D^\times(\mathcal{O}_{v})\iota^{-1}_{v}(\diag(\pi_v,\pi_v))D^\times(\mathcal{O}_{v})]
\end{align*}
are the usual Hecke operators.
It is easy to see that both of $S_{\tau,\psi}(U,A)$ and $S_{\tau,\psi,\chi_{\Sigma}}(U,A)$ have a natural action of $\mathbb{T}^S_A$ (cf.~\cite[Definition 2.2]{Dem}).

Let $U=U^{(p)}\times U_p$ be as above. As explained in~\cite[Section 1]{Taylor1}, if we write
${\rm GL}_2(\A^\infty_M)=\coprod_i D^\times t_i UZ_{D^\times}(\A^\infty_M)$, then
\[
  S_{\tau,\psi}(U,A)\simeq \bigoplus_i W^{(UZ_{D^\times}(\A^\infty_M)\cap t^{-1}_{i}D^\times t_i)/M^\times}_\tau.
\]
The group $(UZ_{D^\times}(\A^\infty_M)\cap t^{-1}_{i}D^\times t_i)/M^\times$ is trivial for all $t$ when
$U$ is sufficiently small (see~\cite[p.623]{Kisin}). Henceforth, we keep this condition until the end of
the section. It follows from this that the functor $W_\tau \mapsto S_{\tau,\psi}(U,A)$ is exact 
(cf.~\cite[Lemma 3.1.4]{Kisin-finite}).

Fix an isomorphism $\iota\colon \bQ_p\simeq \C$. Let
\[
  S_{\tau,\psi}(U_p,A):=\varinjlim_{U^{(p)}}S_{\tau,\psi}\left(U^{(p)}\times U_p,A\right)
\]
where $U^{(p)}$ tends to be small. 

We consider $(\underline{k},\underline{w})=
((k_\sigma)_\sigma,(w_\sigma)_\sigma)\in \Z^{{\rm Hom}(M,\bQ_p)}_{>1}\times \Z^{{\rm Hom}(M,\bQ_p)}$ 
with the property that $k_\sigma+2\omega_\sigma$ is independent of $\sigma$. Set $\delta:=k_\sigma+2w_\sigma-1$ for some (any) $\sigma$.
A special case we will encounter later is $(\underline{k},\underline{w})=(k\cdot\textbf{1},w\cdot\textbf{1})$ with $k\in\Z_{>1}$, $w\in\Z$, and $\mathbf{1}=(1,\dots,1)\in\Z^{{\rm Hom}(M,\bQ_{p})}$, where $\delta=k+2w-1$.
Let $\psi_\C\colon M^\times\bs \A^\times_M\lra \C^\times$ be 
the character defined by 
\begin{equation}\label{psi}
  \psi_\C(z)=\iota\left(N(z_p)^{\delta-1}\psi(z^\infty)\right)N(z_\infty)^{1-\delta}
\end{equation} for 
$z=(z^{(p)},z_p,z_\infty)\in  \A^\times_M$ where the symbol $N$ stands for the norm. 
For each $\sigma\in {\rm Hom}(M,\bQ_p)$, there exists a unique pair of $v|p$ and an embedding $\sigma_v\colon M_v\lra \bQ_p$ such that $\sigma_v|_M=\sigma$. Therefore, we can rewrite $(\sigma)_{\sigma \in {\rm Hom}(M,\bQ_p)}=(\sigma_v)_
{v|p,\ \sigma_v\in {\rm Hom}(M_v,\bQ_p)}$.  

Since $w_\sigma\in \Z$ for each $\sigma$, we can define the algebraic representation 
$\tau_{(\underline{k},\underline{w}),A}$ of 
${\rm GL}_2(\O_p)=\prod_{v|p}{\rm GL}_2(\O_v)$ by 
\begin{equation}\label{alg-rep}
\tau_{(\underline{k},\underline{w}),A}=\bigotimes_{v|p}\bigotimes_{\sigma_v\in {\rm Hom}(M,\bQ_p)}
{\rm Sym}^{k_{\sigma_v}-2}\,{\rm St}_2(A)\otimes {\det}^{w_{\sigma_v}}A
\end{equation}
where ${\rm St}_2$ is the standard representation of dimension two. We often drop the subscript $A$ from 
$\tau_{(\underline{k},\underline{w}),A}$ which should not cause any confusion.
Notice that $\tau_{(\underline{k},\underline{w}),\O_K}\otimes_{\O_K,\iota}\C$ is the algebraic representation of ${\rm GL}_2(\C)$ of
highest weight $(\underline{k},\underline{w})$ so that the center acts by $z\mapsto z^{\delta-1},\ \delta=k_\sigma+2w_\sigma-1$ for 
$z\in \C^\times$. We write 
\begin{equation}\label{n-cl}
S_{\underline{k},\underline{w},\psi}(U_p,A):=
S_{\tau_{(\underline{k},\underline{w})},\psi}(U_p,A)
\end{equation}
for simplicity. 

By~\cite[Corollary 1.2, Lemma 1.3-2]{Taylor1}, we have an isomorphism
\begin{equation}\label{classical}
\left(S_{\underline{k},\underline{w},\psi}(U_p,\O_K)/
S^{{\rm triv}}_{\underline{k},\underline{w},\psi}(U_p,\O_K)\right)\otimes_{\O_K,\iota}\C\simeq \bigoplus_{\pi}\pi^{\infty,p}\otimes \pi^{U_p}_p
\end{equation}
of $D^\times(\A^{p,\infty})$-modules,
where $\pi$ runs over the regular algebraic cuspidal automorphic representations of ${\rm GL}_2(\A_M)$ such that
$\pi$ has central character $\psi_\C$ and $S^{{\rm triv}}_{\underline{k},\underline{w},\psi}(U_p,\O_K)$ is defined to be
zero unless $(\underline{k},\underline{w})=(2\cdot\textbf{1},w\cdot\textbf{1})$, in which case we define it to be the subspace of
$S_{\tau_{(\underline{k},\underline{w})},\psi}(U_p,\O_K)$ consisting of functions that factor through the reduced norm of $D^\times(\A_F)$.

For $\tau_{(\underline{k},\underline{w}),\C}$ as above we can also consider the space of Hilbert cusp forms on ${\rm GL}_2(\A_M)$ of level $U$ and of weight $(\underline{k},\underline{w})$ and
their geometric counterparts (see~\cite[Section 1.5]{Di}).
Thanks to many contributors (see~\cite{Jarvis} and the references there), to each (adelic) Hilbert
Hecke eigen cusp form $f$ of weight $(\underline{k},\underline{w})$ and level $U$, one can attach
an irreducible $p$-adic Galois representation $\rho_{f,\iota_p}\colon G_M\lra {\rm GL}_2(\bQ_p)$ for any rational prime $p$
and fixed isomorphism $\bQ_p\simeq \C$.
The construction also works for $\underline{k}\ge \textbf {1}$ (in the lexicographic order). In particular, in the case when $\underline{k}=\textbf{1}$ (parallel weight one),
its image is finite and it gives rise to an Artin representation $\rho_f\colon G_M\lra {\rm GL}_2(\C)$ (see~\cite{RT}).
Taking a suitable integral lattice, we consider the reduction $\br_{f,p}\colon G_M\lra {\rm GL}_2(\bF_p)$ of
$\rho_{f,p}$ or $\rho_f$ modulo the maximal ideal of $\bZ_p$. 

Let $F$ be a totally real subfield of $M$ such that $[M\colon F]=2$. 
This is possible since $M/\Q$ is a Galois extension of even degree. 
Let $s$ be the generator of ${\rm Gal}(M/F)$. Put $S_M:=\Hom_\Q(M,\bQ)$ and 
let $S^{(1)}_M$ be a subset such that $S_M=S^{(1)}_M\coprod S^{(2)}_M$ where 
$S^{(2)}_M:=S^{(1)}_M\circ s$. 
Further, we identify $S^{(1)}_M$ with $S_F=\Hom_\Q(F,\bQ)$ via $\sigma\mapsto \sigma|_{F}$. 

Let ${}^s f$ be the twist of $f$ by $s$ 
which is defined by the composition of $f$ and 
the automorphism of $D^\times\bs D^\times(\A^\infty_{M})$ induced from $s$. 
We denote by $\omega_{{}^sf}$ the central character, which satisfies $\omega_{{}^sf}={}^s\omega_f$. 
Let $\pi_f$ (resp. $\pi_{{}^sf}$) be the corresponding cuspidal representation attached to 
$f$ (resp. ${}^sf$).  

Changing the subset $S^{(1)}_M$ if necessary, 
we may write $\uk=(k_\tau)_{\tau\in S_M}$ as 
$\uk=(k_{1,\sigma},k_{2,\sigma})_{\sigma\in S^(1)_M}$ 
with $k_{1,\sigma}\ge k_{2,\sigma}$ for each $\sigma\in S^{(1)}_M$, 
where $k_\tau=k_{1,\tau}$ if $\tau\in S^{1}_M$ and $k_\tau=k_{2,\sigma}$ if 
$\tau=\sigma\circ s\in S^{(2)}_M$.

Applying \cite{R},\cite{BR},\cite{Skinner}, we have the following unconditional theorem:
\begin{thm}\label{gal-theta}Keep the notation as above. 
Suppose $\uk=(k_1\cdot\textbf{1},k_2\cdot\textbf{1})$ is in $\Z^{S^{(1)}_M}\times
\Z^{S^{(2)}_M}$ for some integers $k_1\ge k_2\ge 2$ with $k_1\equiv k_2\ {\rm mod}\ 2$. 
For each $\sigma\in S^{(1)}_M$, put $w_\sigma=0$ and 
$w_{\sigma\circ s}=\frac{k_1-k_2}{2}$. 
  Further, suppose $\pi_{{}^sf}$ is not isomorphic to $\pi_f$ (equivalently, $f$ does not come from 
  any Hilbert cusp form on $\GL_2(\A_F)$ via base change) and suppose 
the central character $\omega_f$ satisfies $\omega_{{}^s f}=\omega_f$.   
Then, there exists a Hilbert--Siegel Hecke eigen cusp form $h$ on $\GSp_4(\A_F)$ 
such that the corresponding cuspidal automorphic representation $\Pi=\Pi_h$ 
satisfies the following:
\begin{enumerate}
\item for each infinite place $v\in S_F$ and $\sigma\in S^{(1)}_M$ such that 
$\sigma|_{F}=v$, 
the L-parameter of  $\Pi_v$ is given by 
$\phi_{(-k_{1,\sigma};m_{1,\sigma},m_{2,\sigma})}\otimes |\cdot|$ with 
    \begin{equation*}
m_{1,\sigma}=\ds\frac{k_{1,\sigma}+k_{2,\sigma}-2}{2},\quad
m_{2,\sigma}=\ds\frac{k_{1,\sigma}-k_{2,\sigma}}{2};
    \end{equation*}
\item 
for each $p$ and $\iota_p$, there exists an irreducible $p$-adic Galois representation 
$\rho_{h,\iota_p}\colon G_F\lra \GSp_4(\bQ_p)$ enjoying the following properties:
\begin{enumerate}
\item $\rho_{h,\iota_p}={\rm Ind}^{G_F}_{G_M}\rho_{f,\iota_p}${\rm ;}
\item $\rho_{h,\iota_p}$ satisfies {\rm (1),(2),(3)} of Theorem \ref{gal}{\rm ;}
\item for each finite place $v$ of $F$ dividing $p$, 
$\rho_{h,\iota_p}|_{G_{F_v}}$ is de Rham and it satisfies 
the local-global compatibility up to Frobenius semi-simplification.  Further, it has  
the Hodge-Tate weights as in {\rm (4)} of Theorem \ref{gal}. 
\end{enumerate}
\end{enumerate}
\end{thm}
\begin{proof}The condition on the central character is equivalent to 
$\omega_f$ factoring through the norm map $N_{M/F}$ via a (Hecke) character of 
$\A^\times_F$ (cf. \cite[Lemma 5.2]{Getz}).  Applying \cite[p.251, Theorem 8.6-(1)]{R}, 
the theta lifting from $f\otimes |\det|^{-1}_{\A_M}$ gives an automorphic representation 
$\Pi'$ of $\GSp_4(\A_F)$ such that for each $v\in S_M$, $\Pi'_v$ is a 
(limit) of holomorphic discrete series 
 and $\Pi=\Pi'\otimes |\det|$ is the desired one 
 (the assumption on $\pi_f$ guarantees $\Pi'$ (or $\Pi$) is, in fact, cuspidal). 
For each $\sigma\in S^{(1)}_M$, by the parameter in \cite[p.71, Definition]{BR}, 
the local Langlands parameters of   
$\pi_{f,\sigma}\otimes |\det|^{-1},\ \pi_{f,\sigma\circ s}\otimes |\det|^{-1}$ are given, respectively, by
\begin{align*}
  |z|^{-k_1}&\diag\Big((z/\bar{z})^{\frac{k_{1,\sigma}-1}{2}},
(z/\bar{z})^{-\frac{k_{1,\sigma}-1}{2}}\Big),\\
  |z|^{-k_1}&\diag\Big((z/\bar{z})^{\frac{k_{2,\sigma}-1}{2}},
(z/\bar{z})^{-\frac{k_{2,\sigma}-1}{2}}\Big),
\end{align*}
where $z\in \C\subset W_{M_\sigma}$.
By (\ref{type}), we see the local $L$-parameter of $\Pi'_{\sigma}$ 
(resp. $\Pi_\sigma$) is 
given by $\phi_{(-k_{1,\sigma};m_{1,\sigma},m_{2,\sigma})}$ (resp. 
$\phi_{(-k_{1,\sigma};m_{1,\sigma},m_{2,\sigma})}\otimes|\det|$). 
Therefore, we have the  first claim. 

For the rest, we have only to check the claim (2)-(a),(c) since 
the other one follows from Theorem \ref{gal}. 
The claim (2)-(a) follows from the matching of 
the Satake parameters at unramified places and Chebotarev density.  
By \cite[Theorem 1 with the footnote 7 in p.252]{Skinner}, $\rho_{f,\iota_p}|_{G_{F_v}}$ 
is de Rham and so is $\rho_{h,\iota_p}|_{G_{F_v}}$. Hence we have the claim (2)-(c). 
\end{proof}
\begin{remark}\label{lgc}
The claim (c) of Theorem \ref{gal-theta} above is much stronger than (5) of Theorem \ref{gal} 
when $k_1=k_2$ because as explained in the proof, the local-global compatibility holds 
for any Hilbert modular forms and any finite places by \cite[Theorem 1]{Skinner}. 
\end{remark}

The following proposition is a minor modification of \cite[Corollary 2.12]{BDJ}. 
The proof itself is well-known to experts but we give the details for completeness; we also 
look carefully at what happens to the central characters in the process of the congruence method. 
\begin{proposition}\label{cong}
  Let $p=2$.
Let $M$ be a totally real field of even degree $m$ over $\Q$. 
Let $f$ be a Hilbert Hecke eigen cusp form on ${\rm GL}_2(\A_M)$ of parallel weight one such that
$\br_{f,\iota_2}$ is irreducible. There exists a Hilbert Hecke eigen cusp form $g$ on ${\rm GL}_2(\A_M)$ such that
\begin{enumerate}
\item $\br_{f,\iota_2}\simeq \br_{g,\iota_2}\otimes \psi$ for some continuous character $\psi\colon G_M\lra \bF^\times_2$; 
\item the character corresponding to the central character of $g$ under (\ref{psi}) 
is trivial; 
\item $g$ is of weight $(\underline{2},\underline{0})$ with 
$\underline{2}:=2\cdot \textbf{1}$  $($parallel weight 2$)$. 
Here we write $\underline{0}=(0,\ldots,0)\in \Z^m$.   
\end{enumerate}    
Further, if $\det(\br_{f,\iota_2})$ is trivial, then 
$\psi$ can be taken to be trivial. 
\end{proposition}
\begin{proof}Let $U$ be a sufficiently small open compact subgroup of ${\rm GL}_2(\A_M)$ that fixes $f$.
Since $f$ is of parallel weight one, we have $(\underline{k},\underline{w})=(\underline{1},w\cdot\underline{1})$ for some $w\in\Z$. By twisting 
if necessary, we may assume $w=0$. 
Let $K$ be a finite extension of $\Q_2$ in $\bQ_2$ with ring of integers $\O_K$ such that $K$ contains all Hecke eigenvalues of $f$ for $\mathbb{T}^S_{\O_K}$. Let $\F$ be the residue field of $K$.
As in~\cite[Section 1.5]{Di}, we can view $f$ as a geometric Hilbert modular form over $\O_K$ via a classical Hilbert modular form
associated to $f$ and 
consider its base change $\overline{f}$ to $\F$. After multiplying by a high enough power of the Hasse invariant, we
have another form $\overline{g}_1$ of weight $(\underline{k},\underline{0})$ with 
$\underline{k}=k\cdot \textbf{1},\ k\gg 0$ such that
\begin{enumerate}
\item $\overline{h}=\overline{g}_1$;
\item $\overline{g}_1$ is liftable to
a geometric Hilbert Hecke eigenform $g_1$ over $\O_K$ by enlarging $K$ if necessary;
\item $\br_{g_1,\iota_2}\simeq \br_{f,\iota_2}$.   
\end{enumerate}
We view $g_1$ as a classical Hilbert Hecke eigen cusp form and by the Jacquet--Langlands correspondence,
we have a $2$-adic algebraic quaternionic form $h_1$ in $S_{\tau_{(\underline{k},\underline{0})},\psi}(U,\O_K)$ corresponding to $g_1$;
here $\psi\colon M^\times\bs \A^{\infty}_M\lra \O^\times_K$ is the finite character corresponding to the central
character of $g_1$ under (\ref{psi}). For each finite place $v$ of $M$ lying over $2$,
let $U_{1,v}$ be the subgroup of $U_v$ consisting of all elements congruent to 
$\begin{pmatrix}
1 & \ast  \\
0 & 1 
\end{pmatrix}$. Put $U_{1,2}:=\prod_{v|2}U_{1,v}$. By definition, $\overline{h}_1\in 
S_{\tau_{(\underline{k},\underline{0})},\psi}(U^{(2)}\times U_{1,2},\F)$. 
Since $U_{1,2}$ acts on $\tau_{(\underline{k},\underline{0}),\F}$ unipotently, 
it can be written as a successive extension of the trivial representation of  $U_{1,2}$.
This successive extension commutes with the Hecke action and one can find
a Hecke eigenform $\overline{h}_2\in S_{\tau_{(\underline{2},\underline{0})},\psi}(U^{(2)}\times U_{1,2},\F)$.
Further, notice that $(\F^\times)^2=\F^\times$ and by twisting, we may assume that the reduction 
$\overline{\psi}$ of $\psi$ is trivial. 
Since $U$ is taken to be sufficiently small, there exists a lift 
$h_2\in S_{\tau_{(\underline{2},\underline{0})},\psi^{{\rm triv}}}(U^{(2)}\times U_{1,2},\O_K)$ of $\overline{h}_2$ 
where $\psi^{{\rm triv}}$ stands for the trivial character. In fact, the Hecke eigen system of $\overline{h}_2$ is 
liftable to $\O_K$ by the Deligne--Serre Lemma~\cite[Lemma 6.11]{DS} and 
by multiplicity one, we can recover a Hecke eigenform $h_2$ from the Hecke eigensystem.   
Further $h_2$ is non-trivial since $\ot$ is absolutely irreducible.
Therefore, by~\cite[Lemma 1.3-2]{Taylor1} the form $g$ is taken to be the image of $h_2$ under the Jacquet--Langlands correspondence. 

The remaining claim on the central character is obvious by construction. 
\end{proof}

\subsection{Mod $2$ automorphy}\label{auto}
In this section we prove Theorem~\ref{main1}.
\begin{proof}By assumption and Lemma~\ref{rep-S5},
${\rm Im}(\br_{C,2})$ is of type $S_5(a)$. Applying Proposition~\ref{s5Galois}
there exist a totally real quadratic extension $L/F$ and an absolutely irreducible totally odd 
two-dimensional Galois representation $\ot\colon G_L\lra {\rm GL}_2(\F_4)$ such that
$\br\simeq {\rm Ind}^{G_F}_{G_L}\ot$. Notice that $\ot$ is not equivalent to ${}^\iota \ot$ for 
any lift $\iota$ of the generator of ${\rm Gal}(L/F)$ and ${\rm Im}(\ot)={\rm SL}_2(\F_4)$.  
Since ${\rm SL}_2(\F_4)\simeq A_5$ we have an embedding ${\rm SL}_2(\F_4)
\hookrightarrow {\rm PGL}_2(\C)$. Let $W=W(\F_4)$ be the ring of Witt vectors of 
$\F_4$ and take a sufficiently large finite extension $K/\Q_2$ containing $W$. 
Then by using an explicit realization of the complex representation of $A_5$, 
we have an embedding ${\rm SL}_2(\F_4)
\hookrightarrow {\rm PGL}_2(\O_K)$. Applying Tate's theorem \cite[Section 6]{Serre}, 
we have a representation $\tau'\colon G_L\lra \GL_2(\O_K)$ whose reduction modulo the maximal 
ideal $m_K$ of $\O_K$ is a twist of $\ot\otimes_{\F_4}\O_K/m_K$ by 
a finite character of $G_L$. It follows from this that $\tau'$ is totally odd and 
the projective image is isomorphic to $A_5$. 
By~\cite[Theorem 2]{Sasaki} or \cite[Theorem 0.3]{PS},
$\ot'$ is modular, hence there exists a Hilbert cusp form $f_1$ of parallel weight 1
such that $\br_{f_1,\iota_2}\simeq \ot$.
Since $\ot$ has trivial determinant, $\det(\br_{f_1,\iota_2})$ is trivial. 
Further, applying Proposition~\ref{cong}, there exists a Hilbert cusp form $f$ of
${\rm GL}_2(\A_L)$ of $(\underline{2},\underline{0})$ with trivial central character such that $\br_{f,2}\simeq \ot$.
Let $\pi$ be the cuspidal automorphic representation of ${\rm GL}_2(\A_L)$ corresponding to $f$. For each infinite place $v$ of $L$, the local Langlands parameter of $\pi_{f,v}$ is given by
\[
\phi_{\pi_{f,v}}(z)=
\,\diag\left(\Big(\ds\frac{z}{\overline{z}}\Big)^\frac{1}{2},\Big(\frac{z}{\overline{z}}\Big)^{-\frac{1}{2}}\right).
\]

Applying Theorem \ref{gal-theta}, we have
a cuspidal automorphic representation $\Pi$ of ${\rm GSp}_4(\A_F)$ 
corresponding to ${\rm Ind}^{G_F}_{G_L}\rho_{f,\iota_2}$. 
It follows that the local Langlands parameter of $\Pi_v$ is $\phi_{(1;1,0)}
\otimes |\cdot|^{-1}$. Therefore, $\Pi_v$ is a limit of holomorphic discrete series and 
there exists a Hilbert-Siegel cusp Hecke eigen form $h$ of 
parallel weight 2 generating $\Pi$ as desired.
\end{proof}

\begin{remark}\label{level}
In the course of the proof, we apply Sasaki's (or Pilloni-Stroh's) modularity lifting theorem
which requires raising the level. Therefore, we cannot specify the levels of the Hilbert modular forms
appearing there. This implies that we are also unable to specify the levels of our Hilbert--Siegel modular forms.
However, there is some hope and a substantial expectation for further development of the level-lowering method.
Indeed,~\cite[Conjectures 4.7 and 4.9]{BDJ} indicate that we should be able to take
the Hilbert cusp forms in question to have the correct level corresponding to the Artin conductor of $\br$.
\end{remark}

\section{Examples}\label{sect:examples}
We are looking for monic irreducible polynomials $f\in\Q[x]$ of degree $6$ such that
\begin{enumerate}
  \item there is an isomorphism $\sigma$ from the Galois group of $f$ to the symmetric group $S_{5}$;
  \item for any complex conjugation $c$ in the Galois group of $f$, the element $\sigma(c)\in S_{5}$ has type $(2,2)$.
\end{enumerate}

We shall give two families of such polynomials. 
A general approach to produce them is to compute the sextic resolvent polynomial $f$ of an
irreducible quintic polynomial over $F$ with Galois group $S_5$.
This can be done using symbolic computation, in a straightforward but very tedious way;
also, writing down an explicit form would take considerable space.
To get around this, we apply classical results of Hermite and Klein. 

\subsection{Using generic polynomials}\label{subsect:generic}
Let $F=\Q$. 
Hermite proved that the polynomial $P\in \Q(s,t)[x]$ given by
\[
  P(x) = x^{5}+sx^{3}+tx+t
\]
is generic for $S_{5}$, which means that every $S_{5}$-extension $K/\Q$ is the splitting field of some specialization of $P$, and that the splitting field of $P$ has Galois group $S_{5}$ over $\Q(s,t)$.
A proof appears in~\cite[Proposition 2.3.8]{JLY}.
The discussion in Section 2.4 of the same book contains the following observation:
If $h\in\Q[x]$ is a quintic polynomial with Galois group $S_{5}$, then the Weber sextic resolvent of $h$ is an irreducible sextic with the same splitting field, corresponding to the transitive embedding of $S_{5}$ into $S_{6}$.
We can apply this to a specialization $h$ of $P$ with $s,t\in\Q$, and get the Weber sextic resolvent
\[
  f(x) = (x^{3} + b_{4}x^{2}+b_{2}x+b_{0})^{2} - 2^{10}\disc(h)x,
\]
where $\disc(h)$ is the discriminant of $h$ and the coefficients are given by
\begin{align*}
  b_{0} &= -176s^{2}t^{2}+28s^{4}t+4000st^{2}-s^{6}+320t^{3}\\
  b_{2} &= 3s^{4}-8s^{2}t+240t^{2}\\
  b_{4} &= -3s^{2}-20t.
\end{align*}

Taking different rational values of $s$ and $t$, we can obtain lots of examples of $f$ with property (1) above.

Property (2) is easy to verify once $s$ and $t$ are given, as it is equivalent to $h$ having a unique real root.
We determine the relation between the values of $s$ and $t$ and the number of real roots of $h$ using Sturm's Theorem~\cite[Theorem 2.62]{BPR}, which requires computing the Sturm sequence
\begin{align*}
  h_0(x) &= h(x) = x^5 + sx^3 + tx + t\\ 
  h_1(x) &= h'(x) = 5x^4 + 3sx^2 + t\\ 
  h_2(x) &= - \big(h_0(x)\text{ mod }h_1(x)\big) = -\frac{2}{5}\,sx^3-\frac{4}{5}\,tx-t\\ 
  h_3(x) &= - \big(h_1(x)\text{ mod }h_2(x)\big) = \frac{g_3}{s}\,x^2+\frac{25t}{2s}\,x-t\\ 
  h_4(x) &= - \big(h_2(x)\text{ mod }h_3(x)\big) = \frac{tg_4}{2g_3^2}\,x+\frac{t(5t-s^2)(20t-9s^2)}{g_3^2}\\ 
  h_5(x) &= - \big(h_3(x)\text{ mod }h_4(x)\big) = \frac{g_3^2g_5}{g_4^2},
\end{align*}
where $a(x)\text{ mod }b(x)$ denotes the remainder of the polynomial division of $a(x)$ by $b(x)$, and for future use we define
\begin{align*}
  g_3(s,t) &= 10t-3s^2\\ 
  g_4(s,t) &= 12s^4-88s^2t+125st+160t^2\\ 
  g_5(s,t) &= 108s^5+16s^4t-900s^3t-128s^2t^2+2000st^2+256t^3+3125t^2.
\end{align*}

Given a polynomial $p$ of degree $n$ in $x$ with leading coefficient $a_n$, we define the sign of $p$ at $+\infty$ to be the sign of $a_n$, and the sign of $p$ at $-\infty$ to be $(-1)^n$ times the sign of $a_n$.
For the Sturm sequence $h_0,h_1,\dots,h_5$, we write $V(-\infty)$ for the number of sign changes in the sequence of signs at $-\infty$, and $V(+\infty)$ for the number of sign changes in the sequence of signs at $+\infty$.
Sturm's Theorem then says that the number $r$ of real roots of $h$ is given by
\begin{equation*}
  r = V(-\infty) - V(\infty).
\end{equation*}

We gather some observations that will simplify our analysis:
\begin{enumerate}
  \item If $t<0$ then $g_3(s,t)<0$: clear as it is a sum of negative terms.
  \item If $s<0$ and $t<0$ then $g_4(s,t)>0$: clear as it is a sum of positive terms.
  \item If $s>0$ and $t>0$ then $g_5(s,t)>0$.
    This follows from completing a couple of squares to rewrite:
    \begin{equation*}
      g_5(s,t)=3s(6s^2-25t)^2+16t(s^2-4t)^2+125st^2+3125t^2.
    \end{equation*}
  \item If $s>0$, $t>0$, and $g_3(s,t)>0$ then $g_4(s,t)>0$.
    To see this, note that
    \begin{equation*}
      g_4(s,t)=\frac{4}{3}\,g_3^2(s,t)+\frac{8}{3}\,tg_3(s,t)+125st.
    \end{equation*}
  \item If $s>0$, $t<0$, and $g_4(s,t)<0$, then $g_5(s,t)>0$.
    This follows from
    \begin{multline*}
      g_5(s,t)=\frac{8}{5}\,g_4(s,t)t+108s^5-t\left(\frac{16}{5}\,s^4+900s^3\right)\\ 
      +t^2\left(\frac{64}{5}\,s^2+1800s+3125\right).
    \end{multline*}
  \item If $s<0$, $t>0$, $g_3(s,t)>0$ and $g_4(s,t)>0$, then $g_5(s,t)>0$.
    This follows implicitly from Sturm's Theorem: if $g_5(s,t)<0$ then the sign sequence at $-\infty$ is $-+----$ and the sign sequence at $+\infty$ is $+++-+-$, hence the quantity $V(-\infty)-V(\infty)=-1$.
\end{enumerate}

Table~\ref{tbl:signs} gives the number $r$ of real zeros of $h(x)$ in relation to the signs of the various relevant quantities.
We have used the observations made above to reduce the number of possibilities from $2^5=32$ to $15$:

\begin{table}[ht]
\begin{tabular}{cccccccccr} \toprule
  $s$ & $t$ & $g_3$ & $g_4$ & $g_5$ & signs at $-\infty$ & $V(-\infty)$ & signs at $+\infty$ & $V(+\infty)$ & $r$\\ \midrule
  $+$ & $+$ & $+$   & $+$   & $+$   & $-+++-+$  & $3$          & $++-+++$  & $2$          & $1$\\ 
  $+$ & $+$ & $-$   & $+$   & $+$   & $-++--+$  & $3$          & $++--++$  & $2$          & $1$\\
  $+$ & $+$ & $-$   & $-$   & $+$   & $-++-++$  & $3$          & $++---+$  & $2$          & $1$\\ \midrule 
  $-$ & $+$ & $+$   & $+$   & $+$   & $-+---+$  & $3$          & $+++-++$  & $2$          & $1$\\
  $-$ & $+$ & $+$   & $-$   & $+$   & $-+--++$  & $3$          & $+++--+$  & $2$          & $1$\\
  $-$ & $+$ & $+$   & $-$   & $-$   & $-+--+-$  & $4$          & $+++---$  & $1$          & $3$\\
  $-$ & $+$ & $-$   & $+$   & $+$   & $-+-+-+$  & $5$          & $++++++$  & $0$          & $5$\\
  $-$ & $+$ & $-$   & $+$   & $-$   & $-+-+--$  & $4$          & $+++++-$  & $1$          & $3$\\
  $-$ & $+$ & $-$   & $-$   & $+$   & $-+-+++$  & $3$          & $++++-+$  & $2$          & $1$\\
  $-$ & $+$ & $-$   & $-$   & $-$   & $-+-++-$  & $4$          & $++++--$  & $1$          & $3$\\ \midrule 
  $-$ & $-$ & $-$   & $+$   & $+$   & $-+-+++$  & $3$          & $++++-+$  & $2$          & $1$\\
  $-$ & $-$ & $-$   & $+$   & $-$   & $-+-++-$  & $4$          & $++++--$  & $1$          & $3$\\ \midrule
  $+$ & $-$ & $-$   & $+$   & $+$   & $-++-++$  & $3$          & $++---+$  & $2$          & $1$\\ 
  $+$ & $-$ & $-$   & $+$   & $-$   & $-++-+-$  & $4$          & $++----$  & $1$          & $3$\\ 
  $+$ & $-$ & $-$   & $-$   & $+$   & $-++--+$  & $3$          & $++--++$  & $2$          & $1$\\ 
  \bottomrule
  \medskip
\end{tabular}
  \caption{Details of the inputs to Sturm's Theorem, depending on the possible signs of $s$, $t$, $g_i(s,t)$. The output of Sturm's Theorem, the number $r$ of real roots of the polynomial $h$, appears in the rightmost column.}\label{tbl:signs}
\end{table}

We can summarize the findings in Table~\ref{tbl:signs} as follows:
\begin{proposition}
  The polynomial $h(x)=x^5+sx^3+tx+t$ has exactly $3$ real roots if and only if $g_5(s,t)<0$.
  If $g_5(s,t)>0$, then $h$ has exactly $5$ real roots if and only if $s<0$, $t>0$, $g_3(s,t)<0$, and $g_4(s,t)>0$.
\end{proposition}
The complement of the conditions listed in the Proposition gives the region in the $(s,t)$-plane of interest to us, namely that where $h$ has exactly one real root.
We note that this occurs everywhere in the first quadrant, and almost nowhere in the third quadrant.
The behaviour in the second quadrant is by far the most interesting, as can be seen in Figure~\ref{fig:allquads}.

\begin{figure}[ht]
  \includegraphics[width=0.45\linewidth]{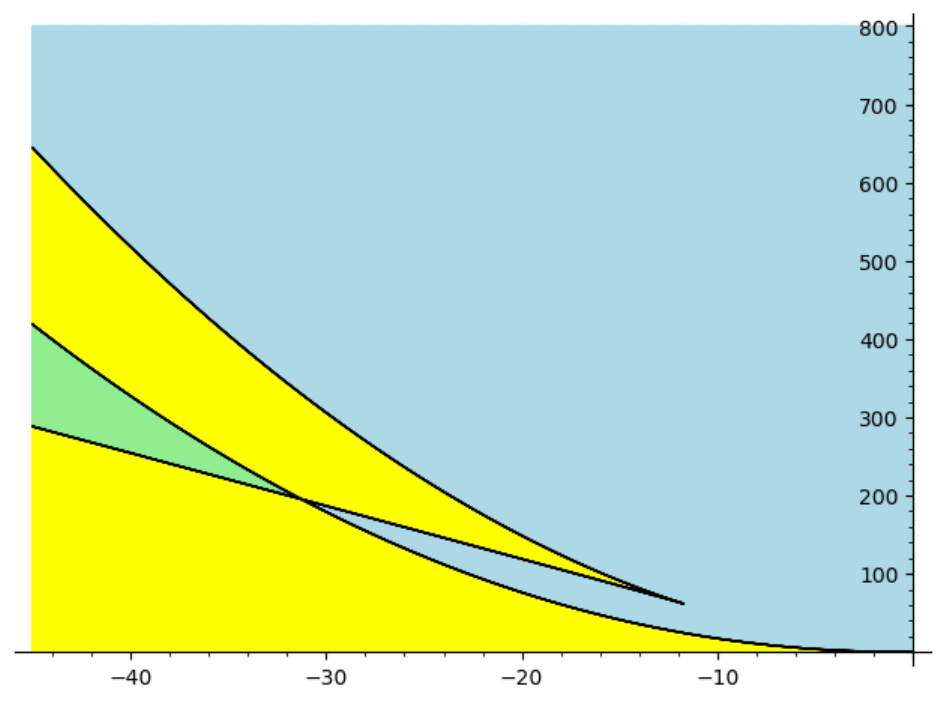}
  \includegraphics[width=0.45\linewidth]{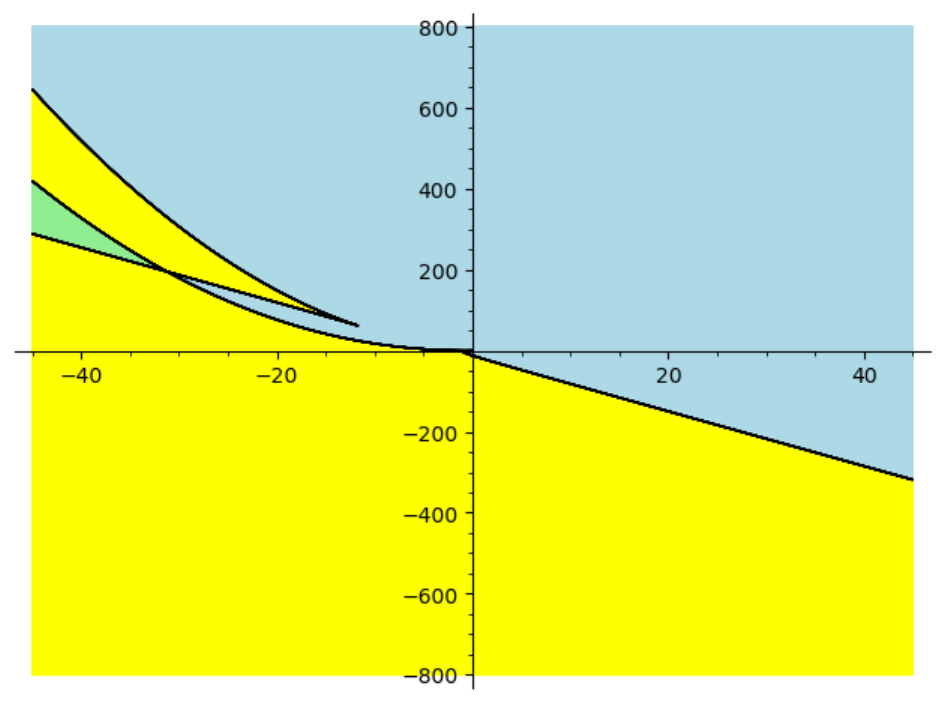}
  \caption{Second quadrant (left) and entire $(s,t)$-plane (right): blue (resp.\ yellow, green) points represent polynomials $h$ having precisely $1$ real root (resp.\ $3$, $5$ real roots). The boundaries of these regions are given by the black curve $g_5(s,t)=0$.}\label{fig:allquads}
\end{figure}

\subsection{Using $5$-torsion on elliptic curves}\label{subsect:division}
Next we consider the quintic polynomials arising from $5$-division points of
elliptic curves.
This has the advantage that an explicit form of the resolvent is well-known: it goes back to work of Klein.
Almost everything that follows is described in~\cite{G-thesis}.

Let $F$ be a totally real field not containing $\sqrt{5}$. 
Let
\[
  E\colon\quad y^2+a_1xy+a_3y=x^3+a_2x^2+a_4x+a_6
\]
be an elliptic curve over $F$ with the property that the mod 5 representation $\br_{E,5}\colon G_F\lra {\rm GL}_2(\F_5)$ is full,\footnote{We say that $\br_{E,p}$ is \emph{full} if ${\rm End}_{\overline{F}}(E)=\Z$ and the image of $\br_{E,p}$ contains ${\rm SL}_2(\F_p)$.} hence surjective.
The projective image of $\br_{E,5}$ is ${\rm PGL}_2(\F_5)$, which is isomorphic to $S_5$. 
In fact, the faithful action of ${\rm PGL}_2(\F_5)$ on the projective line 
$\mathbb{P}^1(\F_5)$ by fractional linear transformations yields 
${\rm PGL}_2(\F_5)\hookrightarrow {\rm Aut}(\mathbb{P}^1(\F_5))\simeq S_6$. 
Since $|{\rm PGL}_2(\F_5)|=120$, by the classification of the subgroups of $S_6$ 
(cf. GAP), ${\rm PGL}_2(\F_5)$ is isomorphic to $S_5$. Further, there are no fixed points in 
this action, hence the image of ${\rm PGL}_2(\F_5)$ is a transitive subgroup 
of $S_6$.  
Under this identification, for each complex conjugation $c$, $\br_{E,p}(c)\sim 
\begin{pmatrix}
0 & 1 \\
1& 0
\end{pmatrix}
$ corresponds to 
an element of $S_6$ of type $(2,2)$. 

The corresponding $S_5$-extension of $F$ is described by a quintic polynomial
\[
  \theta_{E,5}(x)=x^5-40\Delta x^2-5c_4\Delta x-c^2_4\Delta,
\]
(see~\cite[Equation (1.24)]{G-thesis})
where the invariants $\Delta,c_4$ are defined in~\cite[\S{}III.1, p.42]{Sil-Sec}.
The resolvent of $\theta_{E,5}$ is computed as in~\cite[Equation (2.37)]{G-thesis} in conjunction with~\cite[Proposition 3.2.1]{Ku} and it is explicitly given by
\begin{multline*}
  f_{E,6}(x):=x^6+b_2x^5+10b_4x^4+40b_6x^3+80b_8x^2\\ 
  +16(b_2b_8-b_4b_6)x+(-b_2b_4b_6+b^2_2b_8-5b^2_6),
\end{multline*}
where $b_2,b_4,b_6,b_8$ are also defined in~\cite[\S{}III.1, p.42]{Sil-Sec}.
The Galois group of the splitting field $F_{f_{E,6}}$ of $f_{E,6}(x)$ over $F$ in $\bQ$ is a priori embedded in $S_6$ 
by permuting the roots but
it is isomorphic to $S_5$ since $F_{f_{E,6}}=F_{\theta_{E,5}}$ by construction. 
Let us consider the hyperelliptic curve $C=C_E$ defined by $y^2=f_{E,6}(x)$. 
The above argument gives a natural embedding of ${\rm Im}(\br_{C,2})$ into $S_6$.
It is well-known that $\br_{E,5}$ is totally odd and it shows that $\br_{C,2}(c)$ is of type $(2,2)$ for 
each complex conjugation of $G_F$. Further, if $f_{E,6}(x)$ is irreducible, then ${\rm Im}(\br_{C,2})$ is 
transitive in $S_6$ and hence is of type $S_5(a)$. Summing up, we have 
\begin{thm}\label{ex}Assume that $\br_{E,5}$ is full and $f_{E,6}$ is irreducible. 
Then $\br_{C_E,2}$ is automorphic in the sense of Theorem~\ref{main1}.
\end{thm}
We illustrate the method with some explicit examples.
\begin{itemize}
  \item Let us consider the elliptic curve $E\colon y^2+y=x^3+x^2$ over $\Q$, of conductor 43. 
By~\cite[Elliptic curve with label 43.a1]{LMFDB}, the representation $\br_{E,5}\colon G_\Q\lra {\rm GL_2(\F_5)}$ is full.
On the other hand, we have $\theta_{E,5}=x^5+ 1720 x^2+ 3440 x+11008$ and
it is easy to check (e.g.\ using SageMath~\cite{Sage} or Magma~\cite{BCP}) that
\begin{equation}\label{eq:explicit}
  f_{E,6}(x)=x^6+4x^5+40x^3+80x^2+64x+11
\end{equation}
is irreducible over $\Q$ with Galois group isomorphic to $S_5$.
By Theorem~\ref{ex}, $\br_{C_E,2}$ is automorphic.
\item For an example where the ground field is not $\Q$, consider
  \begin{equation*}
    E\colon y^2=x^3-x^2+x\quad\text{ over }F=\Q(\sqrt{2}),
  \end{equation*}
  see~\cite[Elliptic curve 72.1-a3 over number field $\Q(\sqrt{2})$]{LMFDB}.
  Once more, $\br_{E,5}$ is full and the resolvent
    \begin{equation*}
      f_{E,6}(x)=x^6+4x^5+20x^4-80x^2-64x-16
    \end{equation*}
    is irreducible over $F$, so we conclude that $\br_{C_E,2}$ is automorphic.
\end{itemize}

\subsection{Databases of genus 2 curves}

The methods used in~\S\ref{subsect:generic} and~\S\ref{subsect:division} produce infinite families of examples satisfying the conditions of Theorem~\ref{main1}, but they tend to give hyperelliptic curves with large conductors.
For instance, the conductor of the curve~\eqref{eq:explicit} is\footnote{We used the Magma package Genus2Conductor~\cite{DD} by Tim Dokchitser and Christopher Doris for the computation of the conductors listed in this section.}
\begin{equation*}
  \numprint{4786321400000}=2^{6}\cdot 5^{5} \cdot 7 \cdot 43^{4}
\end{equation*}
and the curve corresponding to the choice of parameters $s=0$, $t=-1$ in~\S\ref{subsect:generic} has equation
\begin{equation*}
  y^{2} = x^6 + 40x^5 + 880x^4 + 8960x^3 + 44800x^2 - 3091456x + 102400
\end{equation*}
and conductor
\begin{equation*}
  \numprint{82311610000}= 2^4 \cdot 5^4 \cdot 19^2 \cdot 151^2
\end{equation*}

There are however databases of genus $2$ curves with conductor that is either small or has restricted prime factors.
Currently the most comprehensive are based on~\cite{BSSVY}, which produced:
\begin{itemize}
  \item \numprint{66158} curves with conductor less than \numprint{1000000}, to be found at~\cite[Genus 2 curves over $\Q$]{LMFDB}; from this list, we verified that the curves labelled
  \begin{equation*}
    \begin{array}{rrr}
      \texttt{50000.b.800000.1}
      & \texttt{378125.a.378125.1}
      & \texttt{681472.a.681472.1}\\
      \texttt{64800.c.648000.1}
      & \texttt{382347.a.382347.1}
      & \texttt{703125.b.703125.1}\\
      \texttt{180625.a.903125.1}
      & \texttt{506763.a.506763.1}
  \end{array}
  \end{equation*}
  satisfy the conditions of Theorem~\ref{main1};
  \item \numprint{487493} curves with $5$-smooth discriminant, to be found at~\cite{Sutherland}; we verified that \numprint{4885} of these curves satisfy the conditions of Theorem~\ref{main1}.
\end{itemize}
Both databases provide the curves as global minimal models described by equations of the form $y^2+h(x)\,y=f(x)$.

\bigskip

\bibliographystyle{plain}
\bibliography{refs}

\end{document}